\documentclass[10pt]{amsart}
\usepackage{amsmath}
\usepackage{amsfonts}
\usepackage{amssymb}
\usepackage{graphicx}
\usepackage{amsthm,graphicx,color,yfonts}
\usepackage{pdfsync}
\usepackage{epstopdf}
\usepackage{marginnote}
\usepackage{esint}
\usepackage[title]{appendix}

\usepackage[colorlinks=true]{hyperref}
\hypersetup{linkcolor=red,citecolor=blue,filecolor=dullmagenta,urlcolor=darkblue} 

\numberwithin{equation}{section}

\theoremstyle{plain}
\newtheorem{theorem}{Theorem}[section]

\newtheorem{lemma}[theorem]{Lemma}
\newtheorem*{de-lemma}{Lemma}
\newtheorem{proposition}[theorem]{Proposition}

\theoremstyle{remark}
\newtheorem{remark}[theorem]{Remark}

\theoremstyle{definition}

\DeclareMathOperator{\dv}{div}
\DeclareMathOperator{\curl}{curl}

\newcommand{\R}{\mathbb{R}}

\providecommand{\MR}{\relax\ifhmode\unskip\space\fi MR }

\providecommand{\href}[2]{#2}

\begin{document}

\title[Anisotropic Ginzburg-Landau vortices]{
Entire vortex solutions of negative degree for the anisotropic Ginzburg-Landau system
}

\author[M. Kowalczyk]{Micha{\l} Kowalczyk}
\address{Departamento de Ingenier\'{\i}a Matem\'atica and Centro
de Modelamiento Matem\'atico (UMI 2807 CNRS), Universidad de Chile, Casilla
170 Correo 3, Santiago, Chile}
\email {kowalczy@dim.uchile.cl}

\author[X.Lamy]{Xavier Lamy}
\address{Institut de Math\'ematiques de Toulouse; UMR 5219, Universit\'e de Toulouse; CNRS, UPS IMT, F-31062 Toulouse Cedex 9, France}
\email{Xavier.Lamy@math.univ-toulouse.fr}

\author[P. Smyrenelis]{Panayotis Smyrnelis}
\address{Basque Center for Applied Mathematics, Alameda de Mazarredo 14, 48009 Bilbao, Spain}
\email{psmyrnelis@bcamath.org}

\thanks{M.K. was partially funded by Chilean research grant FONDECYT  1210405 and CMM Conicyt PIA AFB170001. X.L. received  support  from ANR project ANR-18-CE40-0023 and COOPINTER project IEA-297303. Part of this work was done during his visit in the Center of Mathematical Modeling and was partially funded by CMM Conicyt PIA AFB170001 project.
P.M. was supported by the National Science Centre, Poland (Grant No. 2017/26/E/ST1/00817), by REA - Research Executive Agency - Marie Sk{\l}odowska-Curie Program (Individual Fellowship 2018) under Grant No. 832332, by the Basque Government through the BERC 2018-2021 program, by the Spanish State Research Agency through BCAM Severo Ochoa excellence accreditation SEV-2017-0718, and through the project PID2020-114189RB-I00 funded by Agencia Estatal de Investigaci\'on (PID2020-114189RB-I00 / AEI / 10.13039/501100011033). 
}


\begin{abstract} The anisotropic Ginzburg-Landau system
\begin{align*}
\Delta u+\delta \nabla (\dv u) +\delta \curl^*(\curl u)=(|u|^2-1) u, 
\end{align*}
for $u\colon\mathbb R^2\to\mathbb R^2$ and $\delta\in (-1,1)$,
 models the formation of vortices in liquid crystals.
 We prove the existence of entire solutions such that $|u(x)|\to 1$ and $u$ has a prescribed topological degree $d\leq -1$ as $|x|\to\infty$,
for small values of the anisotropy parameter $|\delta| < \delta_0(d)$.
Unlike the isotropic case $\delta=0$, this cannot be reduced to a one-dimensional radial equation. 
We obtain these solutions by minimizing the anisotropic Ginzburg-Landau energy in an appropriate class of equivariant maps, with respect to a  finite  symmetry subgroup.
\end{abstract}

\maketitle	

\section{Introduction}

We study entire solutions $u:\R^2\to\R^2$ of the anisotropic Ginzburg-Landau equation
\begin{align}\label{eq:glani}
\Delta u+\delta \nabla (\dv u) +\delta \curl^*(\curl u)=(|u|^2-1) u, 
\end{align}
where $\delta\in (-1,1)$ is a fixed constant and we define $\curl u=\partial_1 u_2 -\partial_2 u_1$ and its adjoint $\curl^*=(\partial_2,-\partial_1)$. This is the Euler-Lagrange equation corresponding to the anisotropic Ginzburg-Landau energy
\begin{align}\label{eq:Fani}
F(u;\Omega)=\int_{\Omega}\frac 12 |\nabla u|^2 + \frac\delta 2 \left( (\dv u)^2 -(\curl u)^2\right) +\frac 14 (1-|u|^2)^2,\quad\Omega\subset\R^2.
\end{align}
Equation  \eqref{eq:glani} and its associated energy \eqref{eq:Fani}
 arise in the description of 2D point defects in some liquid crystal configurations, such as smectic-$C^*$ thin films \cite{phil1} and nematics close to the Fr\'eedericksz transition \cite{clerc14,clerc1,clerc2}. It has also been proposed as a toy model to understand more complex liquid crystal equations \cite{sternberg1}. While simplified isotropic equations (corresponding here to $\delta=0$) are often used to model liquid crystals, it should be stressed that real liquid crystal materials are always anisotropic, and this reduced symmetry gives rise to nontrivial mathematical challenges.

In the isotropic case $\delta=0$, equation \eqref{eq:glani} is the classical Ginzburg-Landau equation
\begin{equation}\label{eq:gliso}
\Delta  u=(|u|^2-1) u,
\end{equation}
which has been extensively studied due to its physical applications and its mathematical richness -- see the monographs \cite{bethuel1,pacard,SS} and the references therein. 
It is known that for any prescribed degree (or winding number) $d\neq 0$, the isotropic Ginzburg-Landau equation \eqref{eq:gliso} admits a unique solution of the form
\begin{equation}\label{glv}
v_d(re^{i\theta})=\eta_d(r)e^{id\theta},\qquad \eta_d(0)=0,\;\lim_{+\infty}\eta_d=1,\,\eta_d\geq 0,
\end{equation}
where the radial profile $\eta_d$ solves a certain ODE (ordinary differential equation) problem \cite{herve,chen94}. 

In the anisotropic case $\delta\neq 0$ however, the only solutions to \eqref{eq:glani} of the form $u=\eta(r)e^{i(d\theta+\theta_0)}$ are of degree $d=1$, and phase shift $\theta_0\equiv 0\mod\pi/2$ \cite{clerc14}. Therefore, in strong contrast with the isotropic case, the problem of finding
entire solutions $u\colon\R^2\to\R^2$ of \eqref{eq:glani}, such that $|u(x)|\to 1$ and $u$ has a prescribed topological degree $d\neq 1$  as $|x|\to\infty$, cannot be reduced to a one-dimensional ODE. 

Our main result is the existence of solutions of any negative prescribed degree $d$ at infinity, provided the anisotropy $\delta$ is small enough.

\begin{theorem}\label{t:main}
For any $d=-1,-2,\ldots$, there exists $\delta_0(d)>0$  such that for small enough anisotropy $|\delta| <\delta_0(d)$,
there are at least two distinct, smooth entire solutions $u\colon\R^2\to\R^2$  of the anisotropic Ginzburg-Landau equation \eqref{eq:glani} satisfying
\begin{align*}
\int_{\R^2}(1-|u|^2)^2 <2\pi d^2\quad\text{ and }\quad\deg(u;\partial D_r)=d\text{ for }r\gg 1.
\end{align*}
\end{theorem}

\begin{remark}\label{r:rad}
These solutions are, for $\delta\neq 0$, not radially symmetric. This is very different from the isotropic case, where the existence of non-radial solutions is a famous open question. As $\delta\to 0$, the two solutions obtained in Theorem~\ref{t:main} converge to the radial solutions $v_d$ and $iv_d$ \eqref{glv} of the classical Ginzburg-Landau equation \eqref{eq:gliso}. Indeed, they converge to solutions with finite potential energy $\int_{\R^2}(1-|u|^2)^2 <\infty$, of degree $d$ at infinity, and with the symmetry constraint \eqref{eq:dequiv1} for $n=1-d$. That symmetry constraint forces these solutions to have a zero of degree $D$ at the origin, with $|D|\geq |d|$, and they must therefore be radial thanks to a result of Mironescu \cite[Th\'eor\`eme~2]{mironescu-radial}.
\end{remark}

\begin{remark}\label{r:sym}
The two solutions described in the Theorem are distinct modulo the  elementary symmetries of equation \eqref{eq:glani}, which  are the following transformations associated to rotation equivariance, reflection equivariance, and translation invariance: 
\begin{align*}
&u(z) \longrightarrow \tau e^{-i\alpha}u(e^{i\alpha }z)\quad \alpha\in\R,\;\tau\in\lbrace \pm 1\rbrace,\\
& u(z)\longrightarrow \overline{u(\bar z)},\\
& u(z)\longrightarrow u(z+a),\quad a\in\R^2.
\end{align*}
These transformations preserve the equation \eqref{eq:glani} and its associated energy.
Note that in the isotropic case $\delta=0$ we have invariance under any rotation of the variable $u(e^{i\alpha}z)$ and of the target $e^{i\alpha}u(z)$ separately, but for $\delta\neq 0$ this is only true for $\alpha\equiv 0$ modulo $\pi$.
\end{remark}

The solutions we obtain in Theorem~\ref{t:main} are invariant under a well-chosen finite subgroup of the transformations described in Remark~\ref{r:sym}: this is the key to prescribing the degree at infinity. Specifically, we impose the symmetry constraints
\begin{subequations}\label{eq:dequiv}
\begin{equation}\label{eq:dequiv1}
u(e^{i\frac{\pi}{n}}z)=-e^{i\frac{\pi}{n}}u(z)=e^{i\frac{(1-n)\pi}{n}}u(z),
\end{equation}
for some integer $n\geq 2$,
and in addition
\begin{equation}\label{eq:dequiv2}
u(\bar z)=\pm \overline{u(z)},
\end{equation}
\end{subequations}
to distinguish the two different solutions (depending on the sign $\pm$). Thanks to the symmetric criticality principle \cite{palais} (see also \cite[Proposition~7.1]{lamy14bifur} for a simplified version applicable to the present case), minimizing the energy under these symmetry constraints still provides a solution of \eqref{eq:glani}. Moreover, the equivariance \eqref{eq:dequiv1} imposes a degree constraint: if $u$ is continuous and does not vanish on a fixed circle $\partial D_r$, a continuous choice of phase $\psi\colon \R\to\R$ such that $u/|u|(re^{i\theta})=e^{i\psi(\theta)}$  must satisfy 
\begin{align*}
\psi\left(\theta+\frac{\pi}{n}\right)= \psi(\theta) +\frac{(1-n)\pi }{n}+2k_0\pi,
\end{align*}
for some $k_0\in\mathbb Z$ and all $\theta\in\mathbb R$, and iterating this translation yields
$\psi(2\pi)-\psi(0)=2\pi (1-n) + 4 k_0 n\pi $, hence
\begin{align*}
\deg(u;\partial D_r)\equiv (1-n) \mod 2n.
\end{align*}
In the equivalence class $(1-n)+2n \mathbb Z$ the degree with minimal absolute value is $d=1-n$ ($d=-1,-2,\ldots$). This will enable us to conclude that minimizing solutions  have degree $d$ at infinity. 

\begin{remark}\label{r:dpos}
The construction of solutions of degree $d\geq 2$ is an open problem. Note however that from the point of view of physics, the most relevant degrees are $d=\pm 1$ since vortices of degree $|d|\geq 2$ are expected to be unstable.
\end{remark}
With the above discussion in mind, the proof of Theorem~\ref{t:main} will proceed in two steps. First we obtain in Proposition~\ref{p:entire} entire $d$-equivariant solutions with finite potential energy $\int (1-|u|^2)^2<\pi d^2$, by minimizing the energy in large disks $D_R$ and letting $R\to\infty$. This first step works for any value of $\delta\in (-1,1)$ and provides a solution of degree $d_\infty \equiv d$ modulo $2(1-d)$. The second step, in Proposition~\ref{p:deg}, consists in using the minimizing property to show that $d_\infty$ is in fact equal to $d$. This is where we require $|\delta| <\delta_0(d)$. We note that  the value of $\delta_0(d)$, given explicitly in Proposition~\ref{p:deg} can be somewhat improved however whether $\delta_0(d)=1$ or is strictly less than $1$ is an open problem. Theorem~\ref{t:main} follows directly from Propositions~\ref{p:entire} and \ref{p:deg}.

The article is organized as follows. In Section~\ref{s:equiv} we determine the equivariant classes that are compatible with the anisotropic equation \eqref{eq:glani}. In Section~\ref{s:entire} we prove the existence of entire equivariant solutions, while in Section~\ref{s:deg} we characterize their degree for small enough anisotropy. In Section~\ref{s:poho} we derive Pohozaev identity which plays a crucial role in the computations of Section~\ref{s:deg}.

\section{Anisotropic equivariant classes}\label{s:equiv}

Let $G$ and $\Gamma$ be two subgroups of $O(2)$, and let $\mu:G\to \Gamma$ be a group homomorphism. Let also $\Omega\subset \R^2$ be a set invariant by the action of $G$. By definition, a map $u:\Omega \to \R^2$ is $\mu$-equivariant if 
\begin{equation}\label{equiv}
u(gx)=\mu(g)u(x), \ \forall x\in\Omega,\ \forall g\in G.
\end{equation} 
We refer to  \cite{schw} for a discussion of such maps, and to \cite{bates1,bates2} for examples of equivariant solutions. 

An equivariant class is compatible with \eqref{eq:glani} (cf. \cite{palais}), if given any $\mu$-equivariant map $u$, the integrand of \eqref{eq:Fani} is invariant by the action of $G$. Clearly, \eqref{equiv} implies that for every $g\in G$, and $x \in \Omega$, we have
\begin{align*}
&|u(gx)|=|u(x)|, 
\quad
\nabla u(gx)=\mu(g)\circ (\nabla u (x))\circ g^{-1}, \quad |\nabla u(gx)|=|\nabla u (x)|,
\end{align*}
and, letting $\sigma (x)=-x$ denote the antipodal map,
\begin{align*}
(\dv u)(gx)=\pm (\dv u)(x)\quad &\Longleftrightarrow\quad \mu(g)=g \text{ or } \mu(g)=\sigma \circ g,
\\
(\curl u)(gx)=\pm (\curl u)(x)\quad &\Longleftrightarrow\quad \mu(g)=g \text{ or } \mu(g)=\sigma \circ g.
\end{align*}
Thus, the anisotropic equivariant classes of \eqref{eq:glani} are detemined by the homomorphisms $\mu:G\to\Gamma$ satisfying
\begin{equation}\label{condhom}
\forall g\in G: \ \mu(g)=g \text{ or } \mu(g)=\sigma \circ g.
\end{equation}
For instance, the solution $x\mapsto v_1(x/\sqrt{1+\delta})$ of \eqref{eq:glani} (where $v_1$ is defined in \eqref{glv}), is equivariant with respect to the trivial homomorphism $\iota:O(2)\to O(2)$, $\iota(g)=g$, $\forall g\in O(2)$. On the other hand, the solution $x\mapsto i v_1(x/\sqrt{1-\delta})$ of \eqref{eq:glani} is equivariant with respect to the homomorphism $\tilde \iota:O(2)\to O(2)$, such that $\tilde\iota(\mathbf{s})=\sigma \circ \mathbf{s}$ for every reflection $\mathbf {s}\in O(2)$, and $\tilde\iota(\mathbf{r})=\mathbf{r}$ for every rotation $\mathbf{r}\in O(2)$.

Moreover, we point out that the symmetry constraints introduced in \eqref{eq:dequiv} are equivalent to equivariance with respect to the homomorphisms $\mu_d^\pm$, $d=-1,-2,\ldots$ defined below. For every $n=1-d\geq 2$, we denote by $D_{2n}$ the dihedral group generated by the rotation $\mathbf{r}_{2n}$ of angle $\pi/ n$, and the reflection $\mathbf{s}_0$ with respect to the $x_1$ coordinate axis. Setting 
\begin{align*}
\mu_d^+(\mathbf{s}_0)& =\mathbf{s}_0, \quad \mu_d^+(\mathbf{r}_{2n})=\sigma\circ \mathbf{r}_{2n}=\mathbf{r}_{2n}^d,
\\
\mu_d^-(\mathbf{s}_0)&=\sigma\circ\mathbf{s}_0, \quad \mu_d^-(\mathbf{r}_{2n})=\sigma\circ \mathbf{r}_{2n}=\mathbf{r}_{2n}^d,
\end{align*}
we can see that these choices for $\mu_d^\pm(\mathbf{s}_0)$ and $\mu_d^\pm(\mathbf{r}_{2n})$ yield homomorphisms $\mu_d^\pm:D_{2n}\to D_{2n}$, such that $\mu_d^\pm$-equivariance is equivalent to \eqref{eq:dequiv}. Actually, the homomorphisms $\tilde \iota$, and $\mu_d^\pm$ (as well as their restrictions to subgroups), are the only nontrivial homomorphisms satisfying \eqref{condhom}. 

\begin{figure}[h]
\includegraphics[width=.8\textwidth]{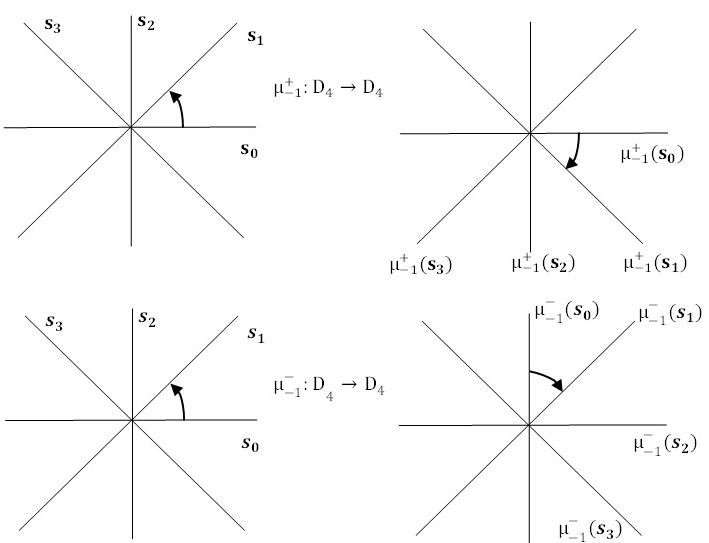}
\caption{For $n=2$ and $d=-1$, the images by the  homomorphisms $\mu_{-1}^\pm:D_4\to \mu_{-1}^\pm(D_4)=D_4$ of the reflections $\mathbf{s}_i:=\mathbf{r}_{4}^i\mathbf{s}_0$ ($i=0,1,2,3$).}
\label{fig1}
\end{figure}

\begin{figure}[h]
\includegraphics[width=.8\textwidth]{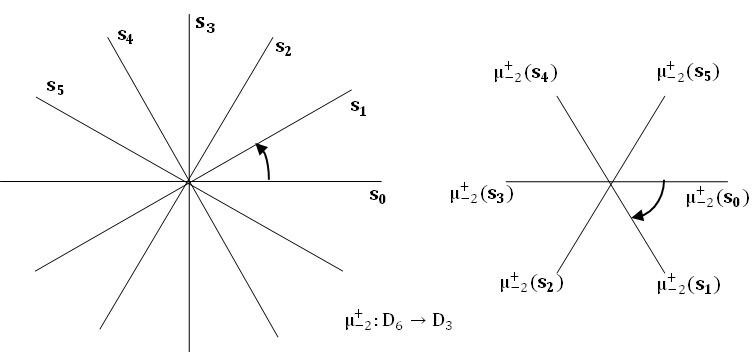}
\caption{For $n=3$ and $d=-2$, the images by the homomorphism $\mu_{-2}^+:D_6\to \mu^+_{-2}(D_6)=D_3$ of the reflections $\mathbf{s}_i:=\mathbf{r}_{6}^i\mathbf{s}_0$ ($i=0,1,2,\ldots,5$).}
\label{fig2}
\end{figure}

Examples of $\mu_d^+$-equivariant maps are $u(re^{i\theta})=f(r)e^{iD\theta}$ for any $D\in d +2(1-d)\mathbb Z$ and real-valued $f(r)$. The transformation $u\mapsto iu$ provides a bijection from $\mu_d^+$-equivariant maps onto $\mu_d^-$-equivariant maps. We also notice that given a reflection ${\mathbf s}\in D_{2n}$, the image by a $\mu_d^\pm$-equivariant map of the reflection line of ${\mathbf s}$, is contained in the reflection line of $\mu_d^\pm ({\mathbf s})$. Figures \ref{fig1} and \ref{fig2} show the images by the homomorphisms $\mu_d^\pm$, of the reflections $ \mathbf{s}_i:=\mathbf{r}_{2n}^i\mathbf{s}_0$ ($i=0,1,\ldots,2n-1$), and illustrate why
$\deg(u;\partial D_r)\equiv d \mod 2n$, for a $\mu_d^\pm$-equivariant map $u$ that does not vanish on the circle $\partial D_r$.

\section{Existence of entire equivariant solutions}\label{s:entire}

The anisotropic Ginzburg-Landau equation \eqref{eq:glani} is preserved by the transformation $(\delta,u)\to (-\delta, iu)$, so  without loss of generality we will restrict ourselves to nonnegative anisotropy $\delta\in [0,1)$. Moreover, thanks to the two-dimensional identity
$\Delta u= \nabla (\dv u) - \curl^*(\curl u)$, we may rewrite \eqref{eq:glani} as
\begin{equation}\label{eq:glapos}
(1-\delta)\Delta u+2\delta \nabla (\dv u) =\nabla W(u),\qquad W(u)=\frac 14 (1-|u|^2)^2,
\end{equation}
with associated energy functional
\begin{equation}\label{eq:glapose}
E(u,\Omega)=\int_\Omega \left[\frac{(1-\delta)}{2}|\nabla u|^2+\delta (\dv u)^2+W(u)\right], \qquad \delta\in[0,1).
\end{equation}
The energy density in \eqref{eq:glapose} differs from the previous one in \eqref{eq:Fani} by a multiple of the null Lagrangian $2\det(\nabla u)=(\dv u)^2+(\curl u)^2 -|\nabla u|^2$. Note that for all $\delta\in (-1,1)$ the left-hand side of the system \eqref{eq:glapos} is an elliptic operator, for which standard $L^p$ and $C^\alpha$ estimates are available.
 In this section we prove the existence of entire $\mu_d^\pm$-equivariant solutions (that is, satisfying the symmetry constraints \eqref{eq:dequiv} for $n=1-d$) of  \eqref{eq:glapos}. 
 
 \begin{proposition}\label{p:entire}
Let $d=-1,-2,\ldots $ and $\delta\in [0,1)$. There exists a smooth $\mu_d^\pm$-equivariant solution $u\colon\R^2\to\R^2$ of the anisotropic Ginzburg-Landau equation \eqref{eq:glapos} such that
\begin{align*}
2\int_{\mathbb R^2} W(u) < \pi d^2,
\end{align*}
$|u(x)|\longrightarrow 1$ as $|x|\to +\infty$,
and $u$ is locally minimizing in the sense that
\begin{align*}
E(u,D_R)\leq E(u+\xi,D_R)\quad\text{for any }\mu_d^\pm\text{-equivariant }\xi\in H^1_0(D_R;\R^2),
\end{align*}
and any centered open disk $D_R$ of finite radius $R>0$.
\end{proposition} 

Proposition~\ref{p:entire} is proved  by minimizing the energy \eqref{eq:glapose} among equivariant maps in the disk $D_R$, with well-chosen boundary condition on $\partial D_R$, and letting $R\to\infty$. The boundary condition on $\partial D_R$ is chosen to minimize the energy among $\mathbb S^1$-valued maps. For a map $u\colon\mathbb S^1\to\mathbb S^1$ (identified with its homogeneous radial extension satisfying $\partial_r u=0$), the energy is given by
\begin{align*}
E(u;\mathbb S^1)=\int_{\mathbb S^1} \left[\frac{1-\delta}{2} |\partial_\theta u|^2 + \delta (\partial_\theta u\cdot e_\theta)^2\right]\, d\theta,\qquad u\in H^1(\mathbb S^1;\mathbb S^1).
\end{align*}
Denoting by $H^1_{equ}(\mathbb S^1;\mathbb S^1)$ 
 the class of $H^1$ maps 
from $\mathbb S^1$ into itself 
that are $\mu_d^\pm$-equivariant, 
we define
\begin{align*}
C_{\delta,d}^\pm =\min_{u\in H^1_{equ}(\mathbb S^1;\mathbb S^1)} E(u;\mathbb S^1).
\end{align*}
The direct method of the calculus of variation ensures that the infimum is indeed attained:
\begin{align}\label{eq:zeta}
C_{\delta,d}^\pm=E(\zeta;\mathbb S^1)\quad\text{for some map }\zeta\in H^1_{equ}(\mathbb S^1;\mathbb S^1),
\end{align}
and we will use this map $\zeta$ as a boundary condition on the circle $\partial D_R$. 

\begin{remark}\label{r:zeta}
The map $\zeta$ can be locally lifted as $\zeta(\theta)=e^{i\psi(\theta)}$ and the real-valued $H^1$ phase $\psi$ solves  the Euler-Lagrange equation
\begin{align*}
\partial_\theta\left[ (1+ \delta\cos(2\theta-2\psi))\partial_\theta\psi \right]=\delta\sin(2\theta-2\psi)(\partial_\theta\psi)^2,
\end{align*} 
which ensures that $\psi$ is smooth (see e.g. \cite[Theorem~4.36]{dacorogna}). Hence $\zeta$ is smooth, and we have
\begin{align*}
C_{\delta,d}^+ < E(e^{id\theta},\mathbb S^1)=\pi d^2,\quad C_{\delta,d}^- < E(ie^{id\theta},\mathbb S^1)=\pi d^2,
\end{align*}
because the map $e^{id\theta}$ (resp. $ie^{id\theta}$) is $\mu_d^+$ (resp. $\mu_d^-$) equivariant but does not satisfy the Euler-Lagrange equation.
\end{remark}

\begin{proof}[Proof of Proposition~\ref{p:entire}]
By the direct method of the calculus of variations, for all $R>0$ there exists a $\mu_d^\pm$-equivariant map $u_R$ minimizing $E(\cdot,D_R)$ among all $\mu_d^\pm$-equivariant maps  $u\in H^1(D_R;\R^2)$ with boundary condition
\begin{align*}
u(Re^{i\theta})=\zeta(\theta)\qquad\forall \theta\in \R,
\end{align*}
where $\zeta$ is defined in \eqref{eq:zeta}. 
The map $u_R$ is a weak solution of the elliptic system \eqref{eq:glapos},  $\nabla W(u_R)\in L^2(D_R)$ and $\zeta$ is smooth, hence from standard elliptic regularity theory (see e.g. \cite[Theorem~4.14]{GM}) we have that $u\in W^{2,2}(D_R)$. 
By Sobolev embedding this implies $u\in C^\alpha(\overline D_R)$ for all $\alpha\in (0,1)$, and bootstrapping Schauder estimates (see e.g. \cite[Theorems~5.20, 5.21]{GM}) one deduces $u\in C^\infty(\overline D_R)$. In particular $u_R$ is regular enough, up to the boundary, to derive Pohozaev's identity (see Section \ref{s:poho})
\begin{align*}
2\int_{D_R} W(u)& = \int_{\mathbb S^1}\left[\frac{1-\delta}{2}|\partial_\theta\zeta|^2 +\delta (\partial_\theta \zeta\cdot e_\theta)^2 \right]
\\&\quad
-R\int_{\partial D_R}\left[
\frac{1+\delta}{2}(\partial_r u\cdot e_r)^2 
+
\frac{1-\delta}{2}(\partial_r u \cdot e_\theta)^2\right]\\
& \leq \int_{\mathbb S^1}\left[\frac{1-\delta}{2}|\partial_\theta\zeta|^2 +\delta (\partial_\theta \zeta\cdot e_\theta)^2 \right],
\end{align*}
and therefore
\begin{align}\label{eq:boundW}
2\int_{D_R}W(u_R) \leq C_{\delta,d}^\pm,
\end{align}
by definition \eqref{eq:zeta} of $\zeta$. 

Fix $x_0\in\R^2$ with $|x_0| < R-2$ and a cut-off function $\eta\in C_c^\infty(D_2(x_0))$ such that $\mathbf 1_{D_1(x_0)}\leq \eta\leq \mathbf 1_{D_2(x_0)}$ and $|\nabla\eta|\leq 2$. Multiplying the equation \eqref{eq:glapos} satisfied by $u_R$ by $\eta^2 u_R$ we obtain 
\begin{align*}
(1-\delta)\int_{D_2(x_0)} \eta^2 |\nabla u_R|^2 & \leq (1+3\delta) \int_{D_2(x_0)} \eta |\nabla\eta|\, |u_R|\, |\nabla u_R| \\
&\quad +  \int_{D_2(x_0)} \eta^2 \left| |u_R|^2-1 \right| \, |u_R|^2\\
&\leq  \frac{1}{2A}(1+3\delta)\int_{D_2(x_0)} \eta^2 |\nabla u_R|^2  + \frac{A}{2}\int_{D_2(x_0)} |\nabla\eta|^2 |u_R|^2 \\
&\quad + \int_{D_2(x_0)} \left( 4 W(u_R) +2 W(u_R)^{\frac 12} \right),
\end{align*}
for any $A>0$. Choosing $A=(1+3\delta)/(1-\delta)$ we deduce
\begin{align*}
\int_{D_2(x_0)} \eta^2 |\nabla u_R|^2&\leq \frac{1+3\delta}{1-\delta}\int_{D_2(x_0)} |\nabla\eta|^2 |u_R|^2 \\
&\quad + \frac{1}{1-\delta}\int_{D_2(x_0)} \left( 4 W(u_R) +2 W(u_R)^{\frac 12} \right)\\
&\leq \frac{1}{1-\delta}\left( C + C \int_{D_R}W(u_R)\right),
\end{align*}
for some absolute constant $C>0$. Thanks to the uniform bound \eqref{eq:boundW} we deduce that $u_R$ is bounded in $H^1(D_1(x_0))$, 
and as a consequence $\nabla W(u_R)$ is bounded in $L^2(D_1(x_0))$, uniformly with respect to $R>2$ and $x_0\in D_{R-2}$. 
From elliptic $L^2$ estimates \cite[Theorem~4.9]{GM} this implies that $u_R$ is bounded in $W^{2,2}(D_1(x_0))\subset C^{\alpha}(D_1(x_0))$ for any $\alpha\in (0,1)$, 
and by Schauder estimates \cite[Theorem~5.20]{GM}  in $C^{2,\alpha}(D_1(x_0))$,
 again  uniformly with respect to $R>2$ and $x_0\in D_{R-2}$.

By Ascoli's theorem and a diagonal argument we may therefore extract a subsequence $R\to +\infty$ such that $u_R$ converges in $C^1_{\mathrm{loc}}(\R^2;\R^2)$ to a solution $u\in C^\infty(\R^2;\R^2)$ of \eqref{eq:glapos}. By construction, $u$ is $\mu_d^\pm$-equivariant, and thanks to \eqref{eq:boundW} it satisfies $2\int_{\R^2}W(u)\leq C^\pm_{\delta,d}<\pi d^2$ (cf Remark~\ref{r:zeta} for the last inequality). 
 The finiteness of the potential energy $\int_{\R^2}W(u)$,  together with the uniform continuity of $u$ (thanks to the above uniform elliptic estimates in $D_1(x_0)$ for any $x_0\in\R^2$), implies that $\lim_{|x|\to\infty}|u(x)|=1$. Finally, for any $R'>R>0$ and $\mu_d^\pm$-equivariant $\xi\in H^1(\R^2;\R^2)$ with support inside $D_R$ we have
\begin{align*}
&E(u;D_R)-E(u+\xi;D_R)\\
&=E(u;D_R)-E(u_{R'};D_R) +E(u_{R'}+\xi;D_R)-E(u+\xi;D_R) \\
&\quad +E(u_{R'};D_R)
-E(u_{R'}+\xi;D_R) \\
&\leq  E(u;D_R)-E(u_{R'};D_R) +E(u_{R'}+\xi;D_R)-E(u+\xi;D_R),
\end{align*}
where we used the fact that $u_{R'}$ is minimizing and $\xi=0$ in $D_{R'}\setminus D_R$. 
The right-hand side converges to 0 as $R'\to\infty$, since $u_{R'}\to u$ in $C^1(\overline D_R)$. This proves the minimizing property of $u$.
\end{proof}

\section{The degree at infinity is equal to $d$ for small anisotropy}\label{s:deg}

So far we have constructed, for any $\delta\in [0,1)$, an entire solution $u$ of the anisotropic Ginzburg-Landau equation \eqref{eq:glapos} with finite potential energy $\int_{\R^2}W(u)<\infty$, whose degree at infinity $d_\infty=\deg(u;\partial D_r)$ for $r\gg 1$, is of the form
\begin{align*}
d_\infty = d + 2(1-d)N\quad\text{ for some }N\in\mathbb Z,
\end{align*}
as explained in the introduction.
Because $d\leq 0$, among all possible values of $d_\infty$ the one with lowest absolute value is precisely $d$. In this section, we use this to conclude that, provided the anisotropy $\delta$ is small enough, $d_\infty$ must in fact be equal to $d$.

\begin{proposition}\label{p:deg}
Let $d=-1,-2,\ldots $, $\delta\in [0,1)$ and $u$ be a $\mu_d^\pm$-equivariant,  locally minimizing  solution of the anisotropic Ginzburg-Landau equation \eqref{eq:glapos}  as described in Proposition \ref{p:entire}.  
Then
\begin{align*}
\deg(u;\partial D_r)=d\text{ for }r\gg 1,\text{ if }0\leq \delta < \delta_0(d)=\min\left( \frac{2-2d}{4d^2 +10 -10d},\frac{2}{\sqrt 3}-1\right).
\end{align*}
\end{proposition}

\begin{proof}[Proof of Proposition~\ref{p:deg}]
The proof is a direct combination of Lemma~\ref{l:lower}, Lemma~\ref{l:Elog} and Lemma~\ref{l:upper} below. Specifically Lemma~\ref{l:lower} provides a lower bound
\begin{align*}
\liminf_{R\to\infty}\frac{E(u,D_R)}{\pi\ln R}\geq (1-\delta)d_\infty^2,
\end{align*}
while Lemmas~\ref{l:Elog} and \ref{l:upper} provide, in two steps, the upper bound
\begin{align*}
\liminf_{R\to\infty}\frac{E(u,D_R)}{\pi\ln R}\leq (1+3\delta)(d^2+|d_\infty -d|),
\end{align*}
whenever $0\leq \delta <  2/\sqrt 3 -1$.
Therefore we must have
\begin{align*}
(1-\delta)d_\infty^2\leq (1+3\delta)(d^2+|d_\infty -d|),
\end{align*}
that is,
\begin{align*}
\delta\geq \frac{d_\infty^2-d^2-|d_\infty -d |}{3d^2+3|d_\infty-d|+d_\infty^2}.
\end{align*}
This implies $d_\infty=d$ as soon as 
\begin{align*}
\delta < \delta_*(d):=
\inf_{d_\infty\in \left(d +  2(1-d)\mathbb Z\right)\setminus\lbrace d\rbrace }
\frac{d_\infty^2-d^2-|d_\infty -d |}{3d^2+3|d_\infty-d|+d_\infty^2}.
\end{align*}
To compute $\delta_*(d)$ we let $d_\infty=d+x$ and consider the function
\begin{align*}
f(x)&=\frac{(d+x)^2-d^2-|x|}{3d^2+3|x|+(d+x)^2}=\frac{x^2 +2d x -|x|}{x^2 + 2d x + 3|x|+ 4d^2},
\end{align*}
whose derivative satisfies
\begin{align*}
(x^2 + 2d x + 3|x|+ 4d^2)^2f'(x)=
\begin{cases}
4(d+x)(2d^2-d+x) &\text{ if }x>0,\\
4(d+x)(2d^2 +d -x) &\text{ if }x<0.
\end{cases}
\end{align*}
In particular we have $f'(x)\leq 0$ for $x\leq -2(1-d)$ and $f'(x)\geq 0$ for $x\geq 2(1-d)$, so
\begin{align*}
\delta_*(d)&=\inf_{x\in 2(1-d)\mathbb Z\setminus \lbrace 0\rbrace}f(x)=\min(f(-2(1-d)),f(2(1-d))) \\
&=f(2(1-d))=\frac{2-2d}{4d^2+10-10d}.
\end{align*}
We conclude that $d_\infty=d$ whenever $\delta<\delta_0(d)$.
\end{proof}

The rest of this section is dedicated to the proof of Lemmas~~\ref{l:lower}, \ref{l:Elog} and \ref{l:upper}. The first is a classical lower bound simply using the fact that $u$ has degree $d_\infty$ at infinity and the energy is larger than $(1-\delta)$ times the isotropic Ginzburg-Landau energy.

\begin{lemma}\label{l:lower}
Let $\delta\in [0,1)$,  $u\colon \R^2\to\R^2$ be a smooth map and $R_0>1$ be such that $|u|\geq 1/2$ on $\partial D_r$ and $\deg(u;\partial D_r)=d_\infty$ for all $r>R_0$. Then 
\begin{align*}
E(u,D_R) \geq (1-\delta)\pi d_\infty^2\ln R + O(1)\quad\text{as }R\to\infty.
\end{align*}
\end{lemma}
\begin{proof}[Proof of Lemma~\ref{l:lower}]
We follow Jerrard's argument \cite{jerrard99} (see also \cite{sandier98}). For $|x| > R_0$ we may
write $u=\rho v$ with $\rho=|u|$ and $v=u/|u|$, and deduce
\begin{align*}
\frac{1-\delta}{2} |\nabla  u|^2 +\delta (\dv u)^2  +W(u) & \geq \frac{1-\delta}{2r^2}|\partial_\theta u|^2 + W(u)\\
&\geq \frac{1-\delta}{2r^2}\rho^2|\partial_\theta v|^2 +\frac{1-\delta}{2r^2}|\partial_\theta\rho|^2  +\frac 14 (1-\rho)^2.
\end{align*}
Letting $h=\max(1-\rho,0)$, we have $\rho^2\geq 1-2h$ and the
 above implies
\begin{align*}
\frac{1-\delta}{2} |\nabla  u|^2 +\delta (\dv u)^2  +W(u) 
&\geq (1-2h)\frac{1-\delta}{2r^2}|\partial_\theta v|^2 +\frac{1-\delta}{2r^2}|\partial_\theta h |^2  +\frac 14 h^2.
\end{align*}
Integrating on $\partial D_r$ for some $r>R_0$ this implies, since $\deg(v)=d_\infty$,
\begin{align}
&\int_{\partial D_r}\left(\frac{1-\delta}{2} |\nabla  u|^2 +\delta (\dv u)^2  +W(u) \right)\nonumber\\
&\geq \frac 1r (1-\delta)\pi d_\infty^2 (1-2h_{max} (r) ) + \int_{\partial D_r} \left(\frac{1-\delta}{2r^2}|\partial_\theta h|^2  +\frac 14 h^2\right),\label{eq:lowerbound1}
\end{align}
where $h_{max}(r)=\max_{\partial D_r}h=h(x_{max})$ for some $x_{max}\in\partial D_r$. To bound the last term from below we set
\begin{align*}
\gamma =\frac 1{r^2} \int_{\partial D_r}|\partial_\theta h|^2,
\end{align*}
so that by Morrey's inequality we have
\begin{align*}
h(x)\geq \frac{h_{max}}{2}\quad\forall x\in\partial D_r\text{ with }|x-x_{max}|\leq c\frac{h_{max}^2}{\gamma},
\end{align*}
for some absolute constant $c>0$. We deduce
\begin{align*}
\int_{\partial D_r} \left(\frac{1-\delta}{2r^2}|\partial_\theta h|^2  +\frac 14 h^2\right)
&\geq \frac{1-\delta}{2}\gamma + c\min \left(r,\frac{h_{max}^2}{\gamma}\right)h_{max}^2,
\end{align*}
for a possibly different absolute constant $c>0$.
Since $r\geq R_0\geq 1$, we deduce
\begin{align*}
\int_{\partial D_r} \left(\frac{1-\delta}{2r^2}|\partial_\theta h|^2  +\frac 14 h^2\right)
&\geq \min\left(c\, h_{max}^2,\inf_{\gamma\geq 0}\left\lbrace \frac{1-\delta}{2}\gamma + c\frac{h_{max}^4}{\gamma}\right\rbrace\right) \\
&\geq c \sqrt{1-\delta}h_{max}^2,
\end{align*}
again for a possibly different absolute constant $c>0$.
Coming back to \eqref{eq:lowerbound1} we deduce
\begin{align*}
&\int_{\partial D_r}\left(\frac{1-\delta}{2} |\nabla  u|^2 +\delta (\dv u)^2  +W(u) \right)\\
&\geq \frac 1r (1-\delta)\pi d_\infty^2 - 2(1-\delta)\pi d_\infty^2 \frac{h_{max}(r)}{r} + c \sqrt{1-\delta}h_{max}(r)^2\\
&\geq \frac 1r (1-\delta)\pi d_\infty^2
+\inf_{h\geq 0}\left\lbrace c \sqrt{1-\delta}h^2
- 2(1-\delta)\pi d_\infty^2 \frac{h}{r}  
\right\rbrace \\
&\geq \frac 1r (1-\delta)\pi d_\infty^2
- c  \frac{d_\infty^4}{r^2},
\end{align*}
for a possibly different absolute constant $c>0$.
Integrating over $r\in [R_0,R]$ implies
\begin{align*}
E(u,D_R\setminus D_{R_0})\geq (1-\delta)\pi d_\infty^2 \ln\frac{R}{R_0} - c \frac{d_\infty^4}{R_0}\quad\text{for all } R\geq R_0,
\end{align*}
and proves the lower bound.
\end{proof}

Lemmas~\ref{l:Elog} and \ref{l:upper} are dedicated to obtaining a similar upper bound on the energy of $u$. The first step is to show that the energy cannot increase faster than logarithmically. In the isotropic case $\delta=0$, such logarithmic energy bound is valid for any solution (not necessarily minimizing) with finite potential energy, as follows from the estimates in \cite{bmr}, which however rely very much on the isotropic structure of the equations and seem unapplicable here. Here we show logarithmic energy growth provided $\delta$ is not too large, by constructing appropriate comparison maps and using the minimizing property of our solutions.

\begin{lemma}\label{l:Elog}
Let $d=-1,-2,\ldots $, $\delta\in [0,1)$  and $u$ be a $\mu_d^\pm$-equivariant solution of the anisotropic Ginzburg-Landau equation \eqref{eq:glapos} described in Proposition \ref{p:entire}.
Then
\begin{align*}
\liminf_{R\to\infty} \frac{E(u,D_R)}{\ln R} < \infty,\text{ provided }0\leq\delta <  \frac{2}{\sqrt 3}-1.
\end{align*}
\end{lemma}

\begin{proof}[Proof of Lemma~\ref{l:Elog}]
Let
\begin{align*}
f(R)&=R \, E(u,\partial D_R),&
g(R)&=R\int_{\partial D_R} (1-|u|^2)^2,&
\sigma(R)&=\sup_{\partial D_R} \left| 1-|u| \right|.
\end{align*}
There exists $R_0\geq 2$ such that for $R\geq R_0$ we have $\sigma(R)\leq 1/2$ and may write
\begin{align*}
u=\alpha\,\rho\, e^{i(d_\infty\theta +\varphi)}\quad\text{in }\R^2\setminus D_{R_0}
\end{align*}
for some real-valued smooth functions $\rho=|u|$, $\varphi$ in $\R^2\setminus D_{R_0}$, 
 and
\begin{align*}
\alpha=\begin{cases}
1 &\text{ if  }u\text{ is }\mu_d^+\text{-equivariant},\\
i &\text{ if  }u\text{ is }\mu_d^-\text{-equivariant},
\end{cases}
\end{align*}
and $d_\infty= d$ modulo $2(1-d)$. The $\mu_d^\pm$-equivariance of $u$ translates into
\begin{align*}
&\rho\left(e^{i\frac{\pi}{1-d}}z\right)=\rho(\bar z)=\rho(z),\quad \varphi\left(e^{i\frac{\pi}{1-d}}z\right)=\varphi(\bar z)=\varphi(z).
\end{align*}
For   all $R>R_0$ we have, with the notation $\partial_\tau =e_\theta\cdot\nabla=r^{-1}\partial_\theta$,
\begin{align*}
& R \int_{\partial D_R} (\partial_\tau\rho)^2 +\left(\frac{d_\infty}{R} + \partial_\tau\varphi\right)^2 
\leq \frac{1}{(1-\sigma(R))^2} R\int_{\partial D_R}|\partial_\tau u|^2,
\end{align*}
hence, for any $\eta\in (0,1)$, using $2R^{-1}d_\infty\partial_\tau\varphi\leq \eta^{-1}R^{-2}d_\infty^2 + \eta\,(\partial_\tau\varphi)^2$,
\begin{align*}
& R \int_{\partial D_R} (\partial_\tau\rho)^2 +(1-\eta)\left( \partial_\tau\varphi\right)^2 
\leq \frac{1}{(1-\sigma(R))^2} R\int_{\partial D_R}|\partial_\tau u|^2
+2\pi\left(1+\frac 1\eta\right)d_\infty^2. 
\end{align*}
From Pohozaev identity (see Section~\ref{s:poho})
\begin{align*}
R\int_{\partial D_R}  \frac{1+\delta}{2}(\partial_\tau u\cdot e_\theta)^2 +
\frac{1-\delta}{2}(\partial_\tau u \cdot e_r)^2  
&= R\int_{\partial D_R}
\frac{1+\delta}{2}(\partial_r u\cdot e_r)^2 
+
\frac{1-\delta}{2}(\partial_r u \cdot e_\theta)^2
\\
& \quad +
2\int_{D_R} W(u) - R\int_{\partial D_R} W(u), 
\end{align*}
we deduce
\begin{align*}
R\int_{\partial D_R} |\partial_\tau u|^2 &\leq
\frac{1+\delta}{1- \delta} {R}\int_{\partial D_R} |\partial_r u|^2
 +\frac{2}{1-\delta} \int_{\mathbb R^2} W(u),
\end{align*}
and therefore
\begin{align}\label{eq:estimrhovarphifR}
& (1-\eta)(1-\sigma(R))^2 R \int_{\partial D_R} (\partial_\tau\rho)^2 +(\partial_\tau\varphi)^2 
\\
& \leq \frac{1+\delta}{2} R\int_{\partial D_R}|\partial_\tau u|^2 
+ \frac{1-\delta}{2}\left( \frac{1+\delta}{1- \delta} R\int_{\partial D_R} |\partial_r u|^2
 +\frac{2}{1-\delta} \int_{\mathbb R^2} W(u)\right)
 +\frac {4\pi}\eta d_\infty^2
 \nonumber\\
& = \frac{1+\delta}{2} R\int_{\partial D_R}|\nabla u|^2 +\int_{\R^2} W(u)
+\frac {4\pi}\eta d_\infty^2\nonumber\\
&\leq \frac{1+\delta}{1-\delta} f(R) +\int_{\R^2} W(u)
+\frac {4\pi}\eta d_\infty^2 .\nonumber
\end{align}
We define real-valued functions $\psi,h$ in $D_R$ solving
\begin{align*}
&\Delta\psi=0\text{ in }D_R,\quad\psi=\varphi\text{ on }\partial D_R,\\
&\Delta h= 0\text{ in }D_R,\quad h=1-\rho\text{ on }\partial D_R.
\end{align*}
From Poisson's formula
\begin{align*}
\psi(re^{i\theta})=\sum_{n\in\mathbb Z} c_n \left(\frac rR\right)^{|n|} e^{in\theta},\quad c_n=\frac{1}{2\pi}\int_0^{2\pi}\varphi(Re^{i\theta})e^{-in\theta}\, d\theta,
\end{align*}
we deduce
\begin{align*}
\int_{D_R}|\nabla \psi|^2 &
=2\pi\sum |n|\, |c_n|^2 \\
&
\leq 
2\pi\sum n^2 |c_n|^2
= R\int_{\partial D_R}(\partial_\tau\varphi)^2,
\end{align*}
and similarly
\begin{align*}
\int_{D_R}|\nabla h|^2 &\leq R\int_{\partial D_R} (\partial_\tau\rho)^2,
\\
\int_{D_R}h^2   &\leq \frac R 2 \int_{\partial D_R} (1-\rho )^2\leq 2 g(R).
\end{align*}
Moreover thanks to the maximum principle we have $|h|\leq \sigma(R)$ in $D_R$. And by uniqueness, the harmonic extensions $h,\psi$ have the same equivariance as $\rho,\varphi$.
Therefore we obtain a $\mu_d^\pm$-equivariant map $\tilde u$ in $D_R$ agreeing with $u$ on $\partial D_R$ by setting
\begin{align*}
\tilde u =\alpha\,\min(r,1-h)\, e^{i(d_\infty\theta +\psi)}\quad\text{in }D_R.
\end{align*}
Denoting by $c$ a generic absolute positive constant (whose value may change from line to line) we have
\begin{align*}
\int_{D_R}|\nabla\tilde u|^2 &\leq \frac{c}{\eta} d_\infty^2 \ln R +\int_{D_R} |\nabla h|^2 + (1+\sigma(R))^2(1+\eta) \int_{D_R} |\nabla\psi|^2\\
&\leq \frac{c}{\eta} d_\infty^2 \ln R +(1+\sigma(R))^2(1+\eta) R\int_{\partial D_R} (\partial_\tau\rho)^2 +(\partial_\tau\varphi)^2\\
&\leq \frac{c}{\eta} d_\infty^2 \ln R 
+\frac{(1+\sigma(R))^2}{(1-\sigma(R))^2}\frac{1+\eta}{1-\eta}\frac{1+\delta}{1-\delta}f(R)
+\int_{\R^2} W(u),
\end{align*}
where we used \eqref{eq:estimrhovarphifR} for the last inequality.
Moreover we have
\begin{align*}
\int_{D_R} W(\tilde u) \leq c + c\int_{D_R}h^2 \leq c +c\, g(R),
\end{align*}
so we deduce 
\begin{align*}
E(u;D_R) &\leq
E(\tilde u;D_R)\leq \frac{1+3\delta}{2}\int_{D_R}|\nabla\tilde u|^2 + \int_{D_R} W(\tilde u)
\\
&\leq \frac{1}{2}\frac{(1+\sigma(R))^2}{(1-\sigma(R))^2}\frac{1+\eta}{1-\eta}\frac{(1+\delta)(1+3\delta)}{1-\delta}f(R)  +\frac{c}{\eta} d_\infty^2 \ln R +\int_{\R^2} W(u)+ c g(R),
\end{align*}
for any $R>R_0$ and $\eta>0$. Setting
\begin{align*}
F(R)=E(u;D_R) +2c \int_{D_R}W(u) =\int_0^R \frac{f(r)+2c g(r)}{r}\, dr,
\end{align*}
we obtain
\begin{align}\label{eq:diffineq}
F(R)&\leq 
c_*(R)\, R\, F'(R) +\frac{c}{\eta} d_\infty^2\ln R + 3c\int_{\R^2} W(u), \\
c_*(R)& =
\frac{1}{2}\frac{(1+\sigma(R))^2}{(1-\sigma(R))^2}\frac{1+\eta}{1-\eta}\frac{(1+\delta)(1+3\delta)}{1-\delta}.\nonumber
\end{align}
Multiplying by $u$ the equation
\begin{align*}
(1-\delta)\Delta u + 2\delta \nabla (\dv u) +(1-|u|^2)u=0,
\end{align*}
and using that $u,\nabla u\in L^\infty$ and $1-|u|^2\in L^2$ we infer that $R\mapsto E(u;B_R)$ grows at most linearly, and therefore
\begin{align}\label{eq:Flin}
\limsup_{R\to\infty}\frac{F(R)}{R}<\infty.
\end{align}
We claim that this linear growth estimate and the differential inequality \eqref{eq:diffineq} imply that
\begin{align}\label{eq:Flog}
\liminf_{R\to\infty}\frac{F(R)}{\ln R} <\infty,\quad\text{provided }\frac{(1+\delta)(1+3\delta)}{1-\delta}<2.
\end{align}
Assume  by contradiction  that $\liminf  F(R)/\ln R =\infty$ and that $(1+\delta)(1+3\delta)/(1-\delta)<2$. Thanks to the latter we can choose $\eta,R_1>0$ so that $c_*(R)<1$ for all $R>R_1$. And since $\liminf_{R\to\infty} F(R)/\ln R =\infty$ we  may further increase $R_1$ and deduce from \eqref{eq:diffineq} the inequality
\begin{align*}
F(R)\leq \hat c RF'(R)\qquad\text{for some }\hat c\in (0,1)\text{ and all }R>R_1,
\end{align*}
But this implies that the function $R\mapsto F(R)/R^{1/\hat c}$ is ultimately nondecreasing and thereby contradicts \eqref{eq:Flin} because $1/\hat c>1$. The conclusion follows from \eqref{eq:Flog} since
 $E(u;B_R)\leq F(R)$.
\end{proof} 

Thanks to Lemma~\ref{l:Elog} we now have a logarithmic bound on the energy, and we may  perform a sharper construction to obtain a more precise upper bound.

\begin{lemma}\label{l:upper}
Let $d=-1,-2,\ldots $ and $u\colon \R^2\to\R^2$ be a $\mu_d^\pm$-equivariant map which locally minimizes the energy $E(u,D_R)$ as described in Proposition \eqref{eq:glapose} and such that
\begin{align*}
\liminf_{R\to\infty} \frac{E(u,D_R)}{\ln R} < \infty.
\end{align*}
Then
\begin{align*}
\liminf_{R\to\infty} \frac{E(u,D_R)}{\ln R} \leq (1+3\delta)\pi \left(d^2 + |d_\infty - d| \right), 
\end{align*}
where $d_\infty=\deg(u;\partial D_r)$ for large enough $r>0$.
\end{lemma}
\begin{proof}[Proof of Lemma~\ref{l:upper}]
The proof follows an argument by Shafrir \cite{shafrir94}.
By assumption we may fix a finite $C_*>\liminf_{R\to\infty} E(u,D_R)/\ln R$, and there exists a sequence $R_0\leq R_k\to +\infty$ such that
\begin{align*}
E(u;\partial D_{R_k})\leq \frac{C_*}{R_k}.
\end{align*}
On $\partial D_{R_k}$  we have $1/2\leq |u|\leq 3/2$  and $\deg(u;\partial D_{R_k})=d_\infty$, so we can write $u=\alpha \rho_k(\theta) e^{i (d_\infty\theta + \varphi_k)}$ for some real valued $H^1$ functions $\rho_k,\varphi_k$, where we take $\alpha=1$ if $u$ is $\mu_d^+$ equivariant, and $\alpha=i$ if $u$ is $\mu_d^-$ equivariant. Equivariance of $u$ translates into 
\begin{align*}
&\rho_k\left(e^{i\frac{\pi}{1-d}}z\right)=\rho_k(\bar z)=\rho_k(z),\quad \varphi_k\left(e^{i\frac{\pi}{1-d}}z\right)=\varphi_k(\bar z)=\varphi_k(z)\quad\forall z\in \partial D_{R_k},
\end{align*} 
and we have, denoting $\partial_\tau=e_\theta\cdot\nabla=r^{-1}\partial_\theta$,
\begin{align*}
R_k\int_{\partial D_{R_k}} \left( |\partial_\tau \rho_k|^2 + |\partial_\tau \varphi_k|^2 \right)\lesssim C_* + d_\infty^2 :=C_{**}.
\end{align*}
Since $\varphi_k$ is determined modulo $2\pi$ we may moreover assume the pointwise bound $|\varphi_k|\lesssim 1+C_{**}$.
Define a $\mu_{d}^\pm$-equivariant map $\tilde u=\alpha\tilde\rho e^{i (d_\infty\theta + \tilde\varphi)}$ in $D_{R_k}\setminus D_{R_k/2}$ by setting
\begin{align*}
\tilde\rho & =1 +\frac{\ln(2r/R_k)}{\ln 2} (\rho_k(\theta) -1)
\\
\tilde\varphi &= \frac{\ln(2r/R_k)}{\ln 2}\varphi_k(\theta).
\end{align*}
With that definition, we have $\tilde u=u$ on $\partial D_{R_k}$, $\tilde u=\alpha e^{id_\infty\theta}$ on $\partial D_{R_k/2}$, and
\begin{align*}
E(\tilde u; D_{R_k}\setminus D_{R_k/2})\lesssim 1+ C_{**}.
\end{align*}
For small $\varepsilon >0$, it is possible to construct a map $u_\varepsilon\colon D_{1/2}\to\mathbb R^2$ which is $\mu_d^{\pm}$-equivariant with $u_\varepsilon=\alpha e^{id_\infty\theta}$ on $\partial D_{1/2}$ and such that
\begin{align*}
E_\varepsilon(u_\varepsilon;D_{1/2})&:=\int_{D_{1/2}} \frac{1-\delta}{2} |\nabla  u|^2 +\delta (\dv u)^2  +\frac{1}{\varepsilon^2}W(u) \\
&\leq (1+3\delta)\pi\left(d^2 + |d_\infty-d| \right)\ln\frac 1\varepsilon + O(1)
\end{align*}
as $\varepsilon\to 0$.
To give a sketch of $u_\varepsilon$'s construction, let $N\geq 0$ be the unique nonnegative integer such that
\begin{align*}
d_\infty - d = 2(1-d)\sigma N,\qquad\sigma\in\lbrace \pm 1\rbrace.
\end{align*}
Then fix $N$ circles with radii $\lambda_j=\frac{j}{4N}$, $j=1,\ldots,N$, and construct $u_\varepsilon$ by putting one vortex of degree $d$ at the origin, and  $2(1-d)$ vortices of degree $\sigma$ equally spaced  along each circle $\partial D_{\lambda_j}$, in a way that respects the $\mu_d^{\pm}$-equivariance. Such configuration will satisfy the claimed energy bound on $u_\varepsilon$. We perform the explicit construction in Lemma~\ref{l:construction} in the $\mu_d^+$-equivariant case and for $\sigma=+1$, the other cases being similar.

 Finally, setting $\tilde u(x)=u_\varepsilon(\varepsilon x)$ for $|x|\leq R_k/2$ and $\varepsilon=1/R_k$, and combining this with the construction of $\tilde u$ in $D_{R_k}\setminus D_{R_k/2}$, we obtain a $\mu_d^{\pm}$-equivariant map $\tilde u$ in $D_{R_k}$ such that $\tilde u=u$ on $\partial D_{R_k}$ and 
\begin{align*}
\frac{1}{\pi\ln R_k}E(\tilde u; D_{R_k})\leq (1+3\delta) \left(d^2 +|d_\infty-d| \right)+ o(1),
\end{align*}
and the minimizing property of $u$ implies the same bound on $E(u;D_{R_k})$.
\end{proof} 

\begin{lemma}\label{l:construction}
For any integers $d\leq -1$, $N\geq 0$ and  $D=d+2(1-d)N$,
there exist $\varepsilon_0=\varepsilon_0(d,N)>0$ and, for $0<\varepsilon<\varepsilon_0$,
 a $\mu_d^+$-equivariant map $u_\varepsilon\colon D_{1/2}\to \R^2$ such that $u((1/2)e^{i\theta})=e^{iD\theta}$ for all $\theta\in \R$, and 
\begin{align*}
\int_{D_{1/2}}\left(\frac 12 |\nabla u_\varepsilon|^2 +\frac {1}{\varepsilon^2}W(u_\varepsilon)\right) \leq 
\pi\left(d^2 + |D-d| \right)\ln\frac 1\varepsilon + C,
\end{align*}
for some constant $C$ depending only on $d,N$.
\end{lemma}
\begin{proof}[Proof of Lemma~\ref{l:construction}]
If $N=0$, i.e. $D=d$, we may simply take $u_\varepsilon(re^{i\theta})=\min(2r/\varepsilon,1)e^{id\theta}$, hence we assume $N\geq 1$.

For $j=1,\ldots, N$ we define radii $\lambda_j=j/4N$, and will construct $u_\varepsilon$ with a vortex of degree $d$ at the origin, and $2(1-d)$ vortices of degree 1 equally spaced along each circle $\partial D_{\lambda_j}$. These degree 1 vortices are centered at 
\begin{align*}
x_{j,k}=e^{ik\frac{\pi}{1-d}}\lambda_j,\qquad k=0,\ldots,2(1-d)-1,\; j=1,\ldots,N.
\end{align*}
We fix a small radius $\rho=\rho(d,N)>0$ such that the disks
$D_{2\rho}(0)$, $D_{2\rho}(x_{j,k})$,
are mutually disjoint. In the set
\begin{align*}
\omega:=D_{\rho}(0)\cup \bigcup_{j,k} D_{\rho}(x_{j,k}),
\end{align*}
we define $u_\varepsilon$  by setting
\begin{align*}
u_\varepsilon(re^{i\theta})&=\min(r/(\rho\varepsilon),1)\,e^{id\theta},\\
u_\varepsilon(x_{j,k}+re^{i\theta})&=(-1)^k \min(r/(\rho\varepsilon),1)\, e^{i\theta},
\end{align*}
for all $r\in [0,\rho)$. Then $u_\varepsilon$ is $\mu_d^+$-equivariant in $\omega$, and
\begin{align*}
\int_{\omega}\left(\frac 12 |\nabla u_\varepsilon|^2 +\frac {1}{\varepsilon^2}W(u_\varepsilon)\right) \leq 
\pi\left(d^2 + |D-d| \right)\ln\frac \rho\varepsilon + C(d,N).
\end{align*}
Since the trace of $u$ on $\partial\omega$ is smooth and $\mathbb S^1$-valued  with  
\begin{align*}
\deg(u_\varepsilon,\partial D_{\rho}(0))+\sum_{j,k}\deg(u_\varepsilon,\partial D_{\rho}(x_{j,k}))=d+2(1-d)N=D,
\end{align*}
there exists a smooth map $u_*\colon D_{1/2}\setminus\omega\to\mathbb S^1$ such that $u_*=u_\varepsilon$ on $\partial\omega$ and $u_*((1/2)e^{i\theta})=e^{iD\theta}$. Moreover, choosing this map $u_*$ minimizing the Dirichlet energy among $\mathbb S^1$-valued maps with these boundary conditions ensures that $u_*$ is $\mu_d^+$-equivariant, because the boundary conditions are equivariant and the minimizer is unique \cite[\S~I.2]{bethuel1}. Hence, setting $u_\varepsilon =u_*$ in $D_{1/2}\setminus\omega$ we obtain a $\mu_d^+$-equivariant map satisfying the conclusion of Lemma~\ref{l:construction}.
\end{proof}

\section{The stress-energy tensor and Pohozaev identities for anisotropic gradient systems}\label{s:poho}

Let $\Omega\subset\R^2$ be an open set. We consider the system
\begin{equation}\label{anisys}
\Delta u+\delta \nabla (\dv u) +\delta \curl^*(\curl u)=\nabla W(u), \  u=(u_1,u_2)\in C^2(\Omega;\R^2),
\end{equation}
with $W \in C^1(\R^2,\R)$, and $\delta\in\R$. Equation \eqref{anisys} can be written as a divergence-free condition, that is,
\begin{equation}\label{divergence free}
  \dv T =  (\nabla u)^\top \big( \Delta u+\delta \nabla (\dv u) +\delta \curl^*(\curl u)-\nabla W(u) \big)=0
\end{equation}
for the stress-energy tensor
\begin{equation*}
T(u,\nabla u)=
\left( \begin{array}{c}
T_{11}~~ T_{12} \\
T_{21} ~~ T_{22} \\
\end{array} \right),
\end{equation*}
with
\begin{align*}
 T_{11}(u,\nabla u)&
 =\frac{1}{2}|\partial_1 u|^2-\frac{1}{2}|\partial_2  u|^2+\frac{\delta}{2}|\nabla u_1|^2-\frac{\delta}{2}|\nabla u_2|^2-W(u),\\
T_{12}(u,\nabla u)&
=\partial_1 u\cdot\partial_2  u+\delta \nabla u_1\cdot \nabla u_2+\delta (\dv u)(\curl u),
\\
T_{21}(u,\nabla u)&=\partial _1 u\cdot\partial_2 u+\delta \nabla u_1\cdot \nabla u_2-\delta (\dv u)(\curl u)
\\
 T_{22}(u,\nabla u)&=\frac{1}{2}|\partial_2 u|^2-\frac{1}{2}|\partial_1 u|^2+\frac{\delta}{2}|\nabla u_2|^2-\frac{\delta}{2}|\nabla u_1|^2-W(u),
\end{align*}
where $\cdot$ stands for the Euclidean inner product. Condition (\ref{divergence free}) is an algebraic fact that follows from a long, but otherwise not difficult computation. In the isotropic case where $\delta=0$, the stress-energy tensor reduces to the symmetric matrix (cf. \cite[Section~3.1]{book} or \cite[Section~5.1]{SS}): 
\begin{equation*}
T(u,\nabla u)=\frac{1}{2}
\left( \begin{array}{cc}
|\partial_1 u|^2 - |\partial_2 u|^2 -2 W(u) & 2 \partial_1 u \!\cdot \partial_2 u \\
2 \partial_1 u \!\cdot\partial_2 u & |\partial_2 u|^2 - |\partial_1 u|^2- 2 W(u) \\
\end{array} \right).
\end{equation*}
Proceeding as in \cite[section 3.6]{book}, we are going to derive Pohozaev identities for \eqref{anisys}, with the help of the stress-energy tensor.
\begin{proposition}
Let $\Omega\subset\R^2$ be an open set, let $D_R(x_0)$ be an open disc of radius $r$ centered at $x_0$, such that 
$\overline{D_r}(x_0)\subset \Omega$, and let $u\in C^2(\Omega;\R^2)$ be a solution of \eqref{anisys}, with $W\in C^1(\R^2;\R)$. For every $x=x_0+re^{i\theta}$, we denote by $(e_r,e_\theta)$ the orthonormal basis 
$(e^{i\theta},ie^{i\theta})$, and by $\partial_r$ (resp. $\partial_\tau$), the partial derivatives with respect to the vector $e_r$ (resp. $e_\theta$).

Then, we have
\begin{align}\label{poz1}
2\int_{D_r(x_0)}W(u)&=r\int_{\partial D_r}\Big(W(u)
+\frac{1+\delta}{2}
(\partial_\tau u\cdot e_\theta)^2
+\frac{1-\delta}{2}(\partial_\tau u\cdot e_r)^2\Big) \\ 
&\quad  
- r\int_{\partial D_r}\Big(
\frac{1+\delta}{2}(\partial_ r u\cdot e_r)^2
+\frac{1-\delta}{2}(\partial_r u\cdot e_\theta)^2\Big),\nonumber
\end{align}
and
\begin{align}\label{poz2}
2\delta\int_{D_r(x_0)}(\dv u)(\curl u)&=r\int_{\partial D_r}\Big((\delta \dv u) (\curl u)\\
&\hspace{5em}-\delta (\nabla (u \cdot e_r))\cdot(\nabla( u\cdot e_\theta))
-\partial_r u\cdot \partial_\tau u\Big).\nonumber
\end{align}
\end{proposition}
\begin{proof}
Without loss of generality we take $x_0=0$. We derive \eqref{poz1} (resp. \eqref{poz2}), by applying the divergence theorem to the vector fields $X=(x_1T_{11}+x_2T_{21},x_1T_{12}+x_2T_{22})$ (resp. $Y=(-x_2T_{11}+x_1T_{21},-x_2T_{12}+x_1T_{22})$ in the disc $D_r(x_0)$, where we have set $x=(x_1,x_2)$. In view of \eqref{divergence free}, one can check that $\dv X=T_{11}+T_{22}=-2W(u)$, while $\dv Y=-T_{12}+T_{21}=-2\delta (\dv u)(\curl u)$. Moreover a long but otherwise not difficult computation gives:
\begin{align*}
X\cdot x&=-r^2\Big(W(u)+\frac{1+\delta}{2}
(\partial_\tau u\cdot e_\theta)^2+\frac{1-\delta}{2}(\partial_\tau u\cdot e_r)^2
\\
&\hspace{7em }
-\frac{1+\delta}{2}(\partial_ r u\cdot e_r)^2-\frac{1-\delta}{2}(\partial_r u\cdot e_\theta)^2\Big),
\end{align*}
and
\begin{align*}
Y\cdot x=-r^2   \Big(\delta (\dv u) (\curl u)-\delta (\nabla (u \cdot e_r))\cdot(\nabla( u\cdot e_\theta))-\partial_r u\cdot \partial_\tau u\Big).
\end{align*}
\end{proof}

\section*{Acknowledgments}

\bibliographystyle{acm}
\bibliography{ref}

\end{document}